\documentclass[reqno]{amsart}
\renewcommand{\texttt}[1]{%
  \begingroup
  \ttfamily
  \begingroup\lccode`~=`/\lowercase{\endgroup\def~}{/\discretionary{}{}{}}%
  \begingroup\lccode`~=`[\lowercase{\endgroup\def~}{[\discretionary{}{}{}}%
  \begingroup\lccode`~=`.\lowercase{\endgroup\def~}{.\discretionary{}{}{}}%
  \begingroup\lccode`~=`\_\lowercase{\endgroup\def~}{\_\discretionary{}{}{}}%
  \catcode`/=\active\catcode`[=\active\catcode`.=\active\catcode`\_=\active
  \scantokens{#1\noexpand}%
  \endgroup
}
\newcommand{\mono}[1]{\texttt{#1}}
\newcommand{\terminology}[1]{\emph{#1}}
\usepackage{amsmath}
\usepackage{amssymb}
\usepackage[T1]{fontenc}
\usepackage[utf8]{inputenc}

\usepackage{array}

\newcommand{\hrulethick} {\noalign{\hrule height 0.11em}}
\newcolumntype{A}{!{\vrule width 0.04em}}
\newcolumntype{B}{!{\vrule width 0.07em}}
\newcolumntype{C}{!{\vrule width 0.11em}}

\newenvironment{smallermatrix}[1]
{\arraycolsep=0.5pt\tiny% change to whatever you need
\array{#1}}
{\endarray}

\usepackage{hyperref}
\usepackage{mathtools}

\newtheorem{theorem}{Theorem}[section]
\newtheorem{lemma}[theorem]{Lemma}
\newtheorem{algorithm}[theorem]{Algorithm}

\theoremstyle{definition}

\newtheorem{example}[theorem]{Example}

\theoremstyle{remark}
\newtheorem{remark}[theorem]{Remark}

\numberwithin{equation}{section}
\allowdisplaybreaks

\usepackage{xypic}\usepackage{mathrsfs}%% End: Author-provided packages
\newcommand{\diff}{\mathop{}\!\mathrm{d}}

\newcommand{\lb}{[}
\newcommand{\rb}{]}

\newcommand{\NN}{\mathbf{N}}
\newcommand{\ZZ}{\mathbf{Z}}
\newcommand{\QQ}{\mathbf{Q}}

\newcommand{\FF}{\mathbf{F}}

\newcommand{\dR}{\mathrm{dR}}

\DeclareMathOperator{\Jac}{Jac}

\newcommand{\lt}{<}
\newcommand{\gt}{>}
\newcommand{\amp}{&}
\begin{document}

\title{Explicit Coleman Integration in Larger Characteristic}

%    Remove any unused author tags.

%    author one information
\author{Alex J. Best}
\address{111 Cummington Mall, Boston MA 02215}
%\curraddr{}
\email{alex.j.best@gmail.com}
\thanks{%
I would like to thank Jennifer Balakrishnan, for suggesting this as something that might be possible, and for many subsequent helpful conversations and comments. Additional thanks are due to Jan Tuitman for remarking that ramified extensions should be avoided, by using \hyperref[lem-coleman-weierstrass-disks-zero]{Lemma~\ref{lem-coleman-weierstrass-disks-zero}}. I have had many interesting conversations with Sachi Hashimoto about Coleman integration. Finally I would like to thank the reviewers for their suggestions. I am grateful for support from the Simons Foundation as part of the Simons Collaboration on Arithmetic Geometry, Number Theory, and Computation \#550023.%
}



%\date{\today}

%\dedicatory{}

% Things to do when copying over
% copy abstract
% make sure no names
% recopy macros sometime

% things to do in submit

% things to do final
% add acknowledge
% fix links

\begin{abstract}
We describe a more efficient algorithm to compute \(p\)-adic Coleman integrals on odd degree hyperelliptic curves for large primes \(p\). The improvements come from using fast linear recurrence techniques when reducing differentials in Monsky-Washnitzer cohomology, a technique introduced by Harvey \cite{Harvey2007} when computing zeta functions. The complexity of our algorithm is quasilinear in \(\sqrt p\) and is polynomial in the genus and precision. We provide timings comparing our implementation with existing approaches.%
\end{abstract}

\maketitle

\section[{Introduction}]{Introduction}\label{subsection-89}

\hypertarget{p-2904}{}%
In 2001, Kedlaya introduced an algorithm for computing the action of Frobenius on the Monsky-Washnitzer cohomology of odd degree hyperelliptic curves over \(\QQ_p\) \cite{Kedlaya2001}. This has been used to compute zeta functions of the reductions modulo \(p\) of such curves, and, starting with the work of Balakrishnan-Bradshaw-Kedlaya \cite{Balakrishnan2010}, to evaluate Coleman integrals between points on them. Computation of Coleman integrals requires more information to be retained throughout the execution of the algorithm than is needed to compute only the way Frobenius acts on cohomology classes, which is all that is needed to compute zeta functions.%
\par
\hypertarget{p-2905}{}%
Harvey \cite{Harvey2007} introduced a variant of Kedlaya's algorithm, its run time in terms of \(p\) alone is \(\widetilde O(\sqrt p) \coloneqq O(\sqrt p \log^k\sqrt p))\) for some \(k\in \ZZ\). In \cite{Balakrishnan2010}  the authors asked if it is possible to use Harvey's techniques when computing Coleman integrals.%
\par
\hypertarget{p-2906}{}%
Here we show that one can obtain the same efficiency improvements in Kedlaya's algorithm as Harvey did, whilst retaining enough information to compute Coleman integrals. Specifically, we obtain the following result:%
\begin{theorem}[{}]\label{thm-coleman-harvey-main}
\hypertarget{p-2907}{}%
    Let \(X/\ZZ_p\) be a genus \(g\), odd degree hyperelliptic curve. Then for the basis \({\{\omega_i = x^i\diff x /2y\}}_{i=0}^{2g-1}\) of \(H^1_\mathrm{dR}(X)\), let \(M\) be the matrix of Frobenius acting on this basis, and \(N \in \NN\) be such that \(X\)  and \(P,Q\in X(\QQ_p)\) are known to precision \(p^N\), assume \( p \gt (2N - 1)(2g + 1)\). Then, if multiplying two \(g\times g\) matrices requires \(O(g^\omega)\) ring operations,  the vector of Coleman integrals \((\int_P^Q \omega_i)_{i = 0}^{2g-1}\) can be computed in time \(\widetilde O\left(g^\omega \sqrt{p} N^{5/2} + N^4 g^4 \log p \right)\) to absolute \(p\)-adic precision \(N - v_p(\det(M-I))\).%
\end{theorem}
As surveyed in \cite{Balakrishnan2010} there are many applications of Coleman integration in arithmetic geometry, notably they are central to the method of Chabauty-Coleman-Kim.
This method has been made explicit in some cases, such as in \cite{Balakrishnan2016} Example 2. There, and in general, when working over number fields it is useful to work only with \(p\) that split. This is an additional condition on \(p\), which often results in having to take larger \(p\), which gives one motivation for the current work.
\par
\hypertarget{p-2908}{}%
In \hyperref[sec-coleman-harvey-setup]{\(\S\)\ref{sec-coleman-harvey-setup}} and \hyperref[sec-coleman-harvey-coleman]{\(\S\)\ref{sec-coleman-harvey-coleman}} we recall the set-up for Coleman integration, and, most importantly, exactly what data is needed to compute Coleman integrals on hyperelliptic curves. In \hyperref[sec-coleman-harvey-reduction]{\(\S\)\ref{sec-coleman-harvey-reduction}} we examine the reduction procedure used by Harvey in more detail. We then come to our main new ideas, creating an appropriate recurrence that computes the data necessary for Coleman integration. In \hyperref[sec-coleman-harvey-linear]{\(\S\)\ref{sec-coleman-harvey-linear}} we introduce a modification of the linear recurrence algorithm used by Harvey, which is specialised to the type of recurrences we obtained. This is useful when computing Coleman integrals between many endpoints simultaneously. In \hyperref[sec-coleman-harvey-algorithm]{\(\S\)\ref{sec-coleman-harvey-algorithm}} we describe the main algorithm in detail. In \hyperref[sec-coleman-harvey-precision-correctness]{\(\S\)\ref{sec-coleman-harvey-precision-correctness}} and \hyperref[sec-coleman-harvey-analysis]{\(\S\)\ref{sec-coleman-harvey-analysis}} we analyse its correctness and complexity. Finally in \hyperref[sec-coleman-harvey-implementation]{\(\S\)\ref{sec-coleman-harvey-implementation}} and \hyperref[sec-coleman-harvey-examples]{\(\S\)\ref{sec-coleman-harvey-examples}} we give some timings and examples obtained with a SageMath/C\mono{++} implementation, showing its practical use.%
\typeout{************************************************}
\typeout{Section 47.2 Set-up and notation}
\typeout{************************************************}
\section[{Set-up and notation}]{Set-up and notation}\label{sec-coleman-harvey-setup}
\hypertarget{p-2909}{}%
Throughout we work with a fixed prime \(p\) and an odd degree hyperelliptic curve  \(X/\ZZ_p\), of genus \(g\ge 1\), given as \(y^2 = Q(x)\) with \(Q(x) \in\ZZ_p\lb x\rb\). Where \(Q(x) = x^{2g+1} + P(x)\) with \(\deg(P) \le 2g\). We assume that the reduction of \(Q(x)\) to \(\FF_p\lb  x\rb\) has no multiple roots. We fix a desired \(p\)-adic precision \(N \ge 1\) such that%
\begin{equation}
p\gt (2N - 1)(2g + 1)\text{.}\label{eqn-assum-p}
\end{equation}
\par
\hypertarget{par-rigid-defs}{}%
Let \(\iota\) denote the hyperelliptic involution, given on the finite affine chart as \((x,y)\mapsto (x,-y)\); the fixed points of this involution are called \terminology{Weierstrass points}.%
\par
\hypertarget{p-2911}{}%
We will make use of several notions from rigid geometry. Points of \(X(\QQ_p)\) which reduce to the same point in \(X_{\FF_p}(\FF_p)\) are said to lie in the same \terminology{residue disk}. A residue disk that contains a Weierstrass point is a \terminology{Weierstrass residue disk}.%
\typeout{************************************************}
\typeout{Section 47.3 Coleman integration}
\typeout{************************************************}
\section[{Coleman integration}]{Coleman integration}\label{sec-coleman-harvey-coleman}
\hypertarget{p-2912}{}%
Coleman integration is a \(p\)-adic (line) integration theory developed by Robert Coleman in the 1980s \cite{Coleman1982,Coleman1988,Coleman1985}. Here we briefly summarise the set-up for this theory (for more precise details, see, for example, \cite{Besser2012}). We also recall the key inputs, which are obtained from Kedlaya's algorithm, for performing explicit Coleman integration on hyperelliptic curves, as described in \cite{Balakrishnan2010}.%
\par
\hypertarget{p-2913}{}%
The setting for Coleman integration as we will be using it is via the Monsky-Washnitzer weak completion of the coordinate ring of the curve minus its Weierstrass points. So, letting \(A = \ZZ_p[x,y,y^{-1}]/(y^2 - Q(x))\), its weak completion is the space \(A^\dagger\) of series \(\sum_{i = -\infty}^\infty R_i(x)y^{-i}\) with \(R_i \in \ZZ_p[x]\), \(\deg R_i \le 2g\) subject to the condition that \(\liminf_{|i| \to \infty} v_p(R_i)/|i| \gt 0\). The \(p\)-power Frobenius on \(\overline A = A/p\)  can be lifted to a function \(\phi\colon A^\dagger \to A^\dagger\) by sending \(x \mapsto x^p\) and \(y \mapsto y^{-p}\sum_{k=0}^{\infty}\binom{-1/2}{k}(\phi(Q(x)) - Q(x)^p)^k /y^{2pk}\). We will consider differentials in \(\Omega_{A^\dagger}^1 =  A^\dagger \diff x \oplus A^\dagger \diff y/(2y\diff y - Q'(x)\diff x))\) with \(\diff\) the exterior derivative%
\begin{equation}
\diff \colon A^\dagger \to \Omega^1_{A^\dagger}; \; \sum_{i=-\infty}^\infty \frac{R_i(x)}{y^i} \mapsto \sum_{i=-\infty}^\infty R'_i(x) y^{-i}\diff x-R_i(x) iy^{-i-1} \diff y\text{.}\label{eqn-exterior-d}
\end{equation}
We will say that \(f\) is a \terminology{primitive} of the exact differential \(\diff f\). We then define the Monsky-Washnitzer cohomology of \(A\) to be \(H^1_{\mathrm{MW}}(\overline A) = \Omega^1_{A^\dagger}\otimes \QQ_p/\diff(A^\dagger\otimes \QQ_p)\). The action of Frobenius and of the hyperelliptic involution can be extended to \(\Omega^1_{A^\dagger}\) and \(H^1_{\mathrm{MW}}(\overline A)\) and the actions of \(\phi\) and \(\iota\) commute. In particular we have an eigenspace decomposition of all of these spaces under \(\iota\) into \terminology{even} and \terminology{odd} parts; the odd part will be denoted with a \(-\) superscript. Let  \(A_\mathrm{loc}(X)\) denote the \(\QQ_p\)-valued functions on \(X(\QQ_p)\) which are given by a power series on each residue disk.%
\begin{theorem}[{Coleman}]\label{thm-coleman-harvey-int}
\hypertarget{p-2914}{}%
There is a unique (up to a global constant of integration) \(\QQ_p\)-linear integration map \(\int\colon \Omega_{A^\dagger}^1\otimes \QQ_p \to A_\mathrm{loc} (X)\) satisfying:\leavevmode%
\begin{enumerate}
\item\hypertarget{li-599}{}\hypertarget{p-2915}{}%
Frobenius equivariance, \(\int \phi^*\omega = \phi^*\int \omega\),%
\item\hypertarget{li-600}{}\hypertarget{p-2916}{}%
the fundamental theorem of calculus, \(\diff \circ \int\) is the identity on \(\Omega_{A^\dagger}^1\otimes \QQ_p\),%
\item\hypertarget{li-601}{}\hypertarget{p-2917}{}%
    and \(\int\circ\diff\) is the natural map \(A^\dagger \to A_{\mathrm{loc}}/(\text{constant functions})\).%
\end{enumerate}
Given points \(P,Q \in X(\QQ_p)\) the definite integral \(\int_P^Q \omega\) is then defined as \(\left(\int \omega\right)(Q) - \left(\int \omega\right)(P)\), which is a well-defined function of \(P,Q\).%
\end{theorem}
\hypertarget{p-2918}{}%
After fixing a basis \(\{\omega_i\}_{i=0}^{2g-1}\) of \(H^1_{\mathrm{MW}}(\overline A)^-=H^1_\dR(X)\), any 1-form of the second kind \(\omega \in \Omega^1_{A^\dagger}\) can be expressed as \(\omega = \diff f + \sum_{i=0}^{2g-1} a_i \omega_i\), \(f\in A^\dagger\), so by \hyperref[thm-coleman-harvey-int]{Theorem~\ref{thm-coleman-harvey-int}} we see that for some \(a_i\in \QQ_p\)%
\begin{equation}
\int_P^Q \omega = f(Q)  - f(P) + \sum_{i=0}^{2g-1} a_i \int_P^Q \omega_i\text{.}\label{eqn-ftc-reduction}
\end{equation}
We can therefore reduce to the case of integrating only the basis differentials \(\omega_i\) and evaluating the primitive \(f\). The complexity of reducing to this case depends on how \(\omega\) is presented. For example, if \(\omega\) has many terms, the total run time can be dominated by finding \(f\) and evaluating \(f(Q) - f(P)\) in the above. So we will focus on computing \(\left\{\int_P^Q\omega_i \right\}_{i=0}^{2g-1}\). In many applications, all that we need to integrate are \(\QQ_p\)-linear combinations of the basis differentials.%
\par
\hypertarget{p-2919}{}%
The work of Balakrishnan-Bradshaw-Kedlaya \cite{Balakrishnan2010} describes how to explicitly compute Coleman integrals for differentials on odd degree hyperelliptic curves. They describe how to reduce the problem of computing general Coleman integrals between two points to that of finding a matrix \(M\) and  \(f_i\in A^\dagger\) such that%
\begin{equation}
\phi^*\omega_i =\diff f_i +  \sum_{j} M_{ij} \omega_j \in \Omega^1_{A^\dagger}\text{.}\label{eqn-frob-basis-decomp}
\end{equation}
\par
\hypertarget{p-2920}{}%
Before stating a form of their algorithm, we recall a useful result which allows us to deal with the difficulties arising when the endpoints of the integral are Weierstrass. This can be problematic, as we need to evaluate primitives as in \hyperref[eqn-ftc-reduction]{(\ref{eqn-ftc-reduction})}; if the endpoints are in Weierstrass residue disks, these power series may not converge.%
\begin{lemma}[{\cite{Balakrishnan2010} Lemma 16}]\label{lem-coleman-weierstrass-disks-zero}
\hypertarget{p-2921}{}%
Let \(P, Q \in X(\QQ_p)\) with \(Q\) Weierstrass and let \(\omega \in \Omega_{A^\dagger}^{1,-}\) be an odd differential without poles at \(P,Q\). Then \(\int^Q_P \omega = \frac 12 \int_P^{\iota(P)} \omega\).%
\par
\hypertarget{p-2922}{}%
In particular, if \(P\) is also a Weierstrass point, then the integral is zero.%
\end{lemma}
\hypertarget{p-2923}{}%
\hyperref[lem-coleman-weierstrass-disks-zero]{Lemma~\ref{lem-coleman-weierstrass-disks-zero}} allows us to express general integrals as linear combinations of integrals between two points in non-Weierstrass residue disks and integrals between two points in the same residue disk (known as \terminology{tiny integrals}). Evaluating tiny integrals uses formal integration of power series, see \cite{Balakrishnan2010} Algorithm 8.%
\par
\hypertarget{p-2924}{}%
Note that \(\infty\) is a Weierstrass point so \hyperref[lem-coleman-weierstrass-disks-zero]{Lemma~\ref{lem-coleman-weierstrass-disks-zero}} applies with \(Q = \infty\); integrals based at \(\infty\) can be rewritten as a linear combination of a tiny integral and an integral between two non-Weierstrass points. Specifically, for a Teichmüller point \(P\), if we know the matrix \(M\) expressing the action of Frobenius on the basis differentials \(\omega_i\), we can use the Frobenius equivariance of the Coleman integral to deduce%
\begin{equation}
\left(\begin{smallmatrix} \vdots \\ \int_{P}^\infty \omega_i \\\vdots \end{smallmatrix}\right) = \frac 12 \left(\begin{smallmatrix} \vdots \\ \int_{P}^{\iota(P)} \omega_i  \\ \vdots \end{smallmatrix}\right) = \frac{(M - I)^{-1}}{2} \left(\begin{smallmatrix}\vdots \\ f_i(P) - f_i(\iota(P)) \\ \vdots \end{smallmatrix}\right)= (M - I)^{-1} \left(\begin{smallmatrix}\vdots \\ f_i(P) \\ \vdots \end{smallmatrix}\right)\label{eqn-coleman-teich-infty-system}
\end{equation}
The last equality holds as we are using odd differentials, so the \(\diff f_i\) must also be odd, so from the expansion of \hyperref[eqn-exterior-d]{(\ref{eqn-exterior-d})} we see that the \(f_i\) must also be odd (up to the constant term, which cancels).%
\par
\hypertarget{p-2925}{}%
So we will fix \(\infty\) as our basepoint and compute only integrals of the form \(\int_P^\infty\omega\); general integrals can be obtained by subtracting two of the above type. We will use the following algorithm, c.f. \cite{Balakrishnan2010} Remark 15:%
\begin{algorithm}[{}]\label{algo-coleman}
\hypertarget{p-2926}{}%
Input: \(P\in X(\QQ_p)\), the matrix of Frobenius \(M\), and if \(P\) is not in a Weierstrass residue disk,  \(\{f_i(P')\}_{ i =0}^{ 2g-1}\) for the unique Teichmüller point \(P'\) in the same residue disk as \(P\), and \(f_i\) as in \hyperref[eqn-frob-basis-decomp]{(\ref{eqn-frob-basis-decomp})}.%
\par
\hypertarget{p-2927}{}%
\noindent Output: \(\left\{\int_P^\infty \omega_i\right\}\) for \(0 \le i \le 2g-1\).\leavevmode%
\begin{enumerate}
\item\hypertarget{item-algo-coleman-step1}{}\hypertarget{p-2928}{}%
If \(P\) is in a Weierstrass residue disk: Let \(P'\) be the Weierstrass point in the same residue disk, so that \(\int_{P'}^\infty \omega_i = 0\) for all \(i\).%
\par
\hypertarget{p-2929}{}%
\noindent Else: Let \(P'\) be the (unique) Teichmüller point in the same residue disk as \(P\). Then compute the vector of \(\int_{P'}^\infty \omega_i\) using \hyperref[eqn-coleman-teich-infty-system]{(\ref{eqn-coleman-teich-infty-system})}.%
\item\hypertarget{item-algo-coleman-step2}{}\hypertarget{p-2930}{}%
For each \(i\), compute the tiny integral \(\int_{P}^{P'} \omega_i\), as in \cite{Balakrishnan2010} Algorithm 8.%
\item\hypertarget{item-algo-coleman-step3}{}\hypertarget{p-2931}{}%
For each \(i\), sum the result of Steps 1 and 2 to get \(\int_{P}^\infty \omega_i= \int_P^{P'} \omega_i + \int_{P'}^\infty \omega_i\).%
\end{enumerate}
\end{algorithm}
\hypertarget{p-2932}{}%
Variants of this algorithm are possible, c.f. \cite{Balakrishnan2010} Algorithm 11. From the version stated above, it is clear that, beyond solving a linear system and computing tiny integrals, the matrix of Frobenius and evaluations of the primitives \(f_i\) at Teichmüller points in non-Weierstrass residue disks are all the input data that is needed to compute arbitrary Coleman integrals. We shall refer to this data as the \terminology{Coleman data}. To compute Coleman integrals efficiently, we require an efficient way of computing this data, possibly for several disks of interest.%
\begin{remark}[]\label{remark-82}
\hypertarget{p-2933}{}%
We do not need to compute the \(f_i\) themselves to compute integrals, only evaluations at Teichmüller points in prescribed non-Weierstrass residue disks. This simplification is key to our ability to write down a suitable recurrence. Moreover, once the Coleman data is computed, it can be saved and will not need to be recomputed if integrals between other points in the same residue disks are required.%
\end{remark}
\typeout{************************************************}
\typeout{Section 47.4 Reductions in cohomology}
\typeout{************************************************}
\section[{Reductions in cohomology}]{Reductions in cohomology}\label{sec-coleman-harvey-reduction}
\typeout{************************************************}
\typeout{Subsection 47.4.1 Kedlaya's algorithm and Harvey's work}
\typeout{************************************************}
\subsection[{Kedlaya's algorithm and Harvey's work}]{Kedlaya's algorithm and Harvey's work}\label{subsection-102}
\hypertarget{p-2934}{}%
Kedlaya's algorithm computes the action of Frobenius on Monsky-Washnitzer cohomology up to a specified precision. The general strategy is to begin with a finite \(p\)-adic approximation of \(\phi^* \omega\) as a (Laurent) polynomial in \(x\) and \(y\) multiplied by the differential \(\diff x/2y\). This is reduced step-by-step via cohomologous differentials of lower polynomial degree, by subtracting appropriate exact forms \(\diff g\) for polynomials \(g\). This process is continued until one is left with a \(\QQ_p\)-linear combination of basis elements, and we have an expression of the form \hyperref[eqn-frob-basis-decomp]{(\ref{eqn-frob-basis-decomp})}. For a given basis \(\{\omega_i\}\) of \(H^1_{\mathrm{MW}}(\overline A)^-\), writing each \(\phi^*\omega_i\) in terms of this basis results in a matrix of Frobenius acting on \(H^1_{\mathrm{MW}}(\overline A)^-\).%
\par
\hypertarget{p-2935}{}%
The innovation in \cite{Harvey2007} is to express the reduction process as a linear recurrence, where the coefficients are linear polynomials in the index of the recurrence. A term several steps later in such recurrences can then be found more efficiently than the straightforward sequential approach, via the algorithm of Bostan-Gaudry-Schost \cite{Bostan2007} Theorem 15. Here we also ultimately appeal to these methods, and so we must examine in more detail the polynomials \(g\) used in the reduction steps. We will describe the sum of the evaluations of these \(g\) at points of interest as a linear recurrence, so that they may be computed along with the reductions.%
\par
\hypertarget{p-2936}{}%
We use the basis of \(H^1_{\mathrm{MW}}(\overline A)^-\) consisting of \(\omega_i = x^i \diff x/2y\) for \(0 \le i \le 2g - 1\). This differs by a factor of \(2\) from the basis used by Harvey and Kedlaya; this choice reduces the number of \(2\)'s appearing in our formulae and so appears more natural here. Changing the basis by a scalar multiple has no effect on the matrix of Frobenius, only the exact differentials. An approximation to \(\phi^* \omega_i\) is given in \cite{Harvey2007} (4.1) by letting \(C_{j,r}\) be the coefficient of \(x^r\) in \(Q(x)^j\) and \(B_{j,r} = p \phi(C_{j,r}) \sum_{k=j}^{N-1} (-1)^{k+j} \binom{-1/2}{k} \binom{k}{j} \in \ZZ_p\) so that%
\begin{equation}
\phi^* \omega_i \equiv\sum_{j=0}^{N-1} \sum_{r=0}^{(2g+1)j} B_{j,r} x^{p(i+r+1) - 1}y^{-p(2j+1) + 1} \frac{\diff x}{2y}\pmod{p^N}\text{.}\label{eqn-approx-frobomegai}
\end{equation}
 In \hyperref[eqn-approx-frobomegai]{(\ref{eqn-approx-frobomegai})}, there are only \((2g + 1)\frac{N(N-1)}{2} + N\) terms in total and the exponents of \(x\) and \(y\) that appear are always congruent to \(-1\) or \(1\) mod \(p\) respectively.%
\par
\hypertarget{p-2937}{}%
As in \cite{Harvey2007} Section 5, we work with finite-dimensional vector spaces over \(\QQ_p\)%
\begin{equation}
W_{s,t} = \left\{ f(x) x^s y^{-2t}\frac{\diff x}{2y}: \deg f \le 2g \right\}  = \left\langle x^i x^s y^{-2t}\frac{\diff x }{2y}\right\rangle_{i=0}^{2g}\text{,}\label{eqn-basis-Wst}
\end{equation}
for \(s \ge -1,\,t\ge 0\), where, in addition, we restrict \(W_{-1,t}\) to be the subspace of the above for which the coefficient of \(x^{-1}\) is zero (i.e.\@ for which \(f(0) = 0\)).%
\par
\hypertarget{p-2938}{}%
Notice that \(W_{-1,0}\) is naturally identified with \(H^1_{\mathrm{MW}}(\overline A)^-\) with the basis chosen above, so that \(\omega_i\) is the \(i\)th basis element of \(W_{-1,0}\). In order to derive an expression for \(\phi^* \omega_i\) as a linear combination of the other basis elements, we begin with the approximation of \(\phi^* \omega_i\) from \hyperref[eqn-approx-frobomegai]{(\ref{eqn-approx-frobomegai})}. Then starting with the terms of highest degree in \(x\), which are each inside of some \(W_{s,t}\) we reduce ``horizontally'', finding a cohomologous element of \(W_{s-1,t}\) by subtracting an appropriate exact differential. This process is repeated until \(s = -1\), but whenever we reach a space \(W_{s,t}\) containing a term from \hyperref[eqn-approx-frobomegai]{(\ref{eqn-approx-frobomegai})}, we add it to the current differential under consideration. We do this for each  \(t\) appearing as an exponent for a monomial in the original approximation, and for each such \(t\) we obtain an element of \(W_{-1,t}\). We then reduce ``vertically'', beginning with the largest \(t\) we have, we subtract appropriate exact differentials to reduce the element of each \(W_{-1,t}\) to a cohomologous one in \(W_{-1, t-1}\) while \(t \ge 1\). This is continued until we have reduced everything to the space \(W_{-1,0}\), and we have  obtained a linear combination of the basis differentials that is cohomologous to \(\phi^*\omega_i\) up to the specified precision.%
\par
\hypertarget{p-2939}{}%
Note that many horizontal \terminology{rows} will not be considered at all. When \(p\) is large enough, most steps simply involve reducing terms we already have, as there are comparatively few terms in the \hyperref[eqn-approx-frobomegai]{(\ref{eqn-approx-frobomegai})} compared to the total degree. Doing multiple reduction steps quickly will therefore improve the run time of this procedure, even though we have to add new terms occasionally. This is where Harvey applies linear recurrence techniques to speed up this reduction process. We now state the reductions we will use; compared to  \cite{Harvey2007} (5.2) and (5.3) we must be more explicit about the exact form we are subtracting, as this data is important for us.%
\typeout{************************************************}
\typeout{Subsection 47.4.2 Horizontal reduction}
\typeout{************************************************}
\subsection[{Horizontal reduction}]{Horizontal reduction}\label{subsection-103}
\hypertarget{p-2940}{}%
To reduce horizontally from \(W_{s,t}\) to \(W_{s-1,t}\), we express the highest order basis element \(x^{2g} x^{s} y^{-2t}\diff x /2y\in W_{s,t}\) as a cohomologous term in \(W_{s-1,t}\). The other basis elements are naturally basis elements for \(W_{s-1,t}\) just with their indices shifted by 1.%
\begin{lemma}[{Horizontal reduction}]\label{lem-harvey-horizontal-reduction}
\hypertarget{p-2941}{}%
We have%
\begin{gather}
x^{2g}x^{s} y^{-2t}\frac{\diff x}{2y} - \frac{-1}{(2t - 1)(2g + 1) -2s} \diff(x^s y^{-2t + 1})\notag\\
= \frac{2sP(x) - (2t - 1 )xP'(x)}{(2t - 1)(2g + 1) -2s} x^{s-1} y^{-2t}\frac{\diff x}{2y} \in W_{s-1,t}\text{.}\label{mrow-254}
\end{gather}
\end{lemma}
\begin{proof}\hypertarget{proof-307}{}
\hypertarget{p-2942}{}%
We directly compute%
\begin{align}
\diff(x^s y^{-2t + 1}) = \amp sx^{s-1} y^{-2t + 1} \diff x + (-2t + 1 ) x^s y^{-2t} \diff y\notag\\
= \amp\left( sx^{s-1} y^{-2t + 1} + \frac12 (-2t + 1 ) x^s y^{-2t-1}Q'(x)\right) \diff x\notag\\
= \amp\left(2 s Q(x) - (2t - 1 ) xQ'(x)\right) x^{s-1} y^{-2t}\frac{\diff x}{2y}\notag\\
= \amp\left(2s - (2t - 1 ) (2g+1)  \right)x^{2g+1} x^{s-1} y^{-2t}\frac{\diff x}{2y}\notag\\
\amp + \left(2sP(x) - (2t - 1 )xP'(x) \right) x^{s-1} y^{-2t}\frac{\diff x}{2y}\text{.}\label{mrow-259}
\end{align}
Therefore, by subtracting \(\frac{1}{2s - (2t - 1)(2g + 1)} \diff(x^s y^{-2t + 1})\) from \(x^{2g}x^{s} y^{-2t}\diff x/2y\), the remaining terms are all as stated, and of lower degree.%
\end{proof}
\typeout{************************************************}
\typeout{Subsection 47.4.3 Vertical reduction}
\typeout{************************************************}
\subsection[{Vertical reduction}]{Vertical reduction}\label{subsection-104}
\hypertarget{p-2943}{}%
To reduce vertically from \(W_{-1,t}\) to \(W_{-1,t-1}\), we express the \(2g\) basis elements \(x^{i} y^{-2t}\diff x /2y\in W_{-1,t}\) as cohomologous terms in \(W_{-1,t-1}\).%
\begin{lemma}[{Vertical reduction}]\label{lem-harvey-vertical-reduction}
\hypertarget{p-2944}{}%
Let \(R_i(x), S_i(x)\in \ZZ_p(x)\) be such that \(x^i = R_i(x) Q(x) + S_i(x)Q'(x)\) with \(\deg R_i \le 2g-1\), \(\deg S_i\le 2g\). Then%
\begin{equation*}
x^{i} y^{-2t}\frac{\diff x}{2y} -\frac{-1}{2t - 1} \diff(S_i(x)y^{-2t + 1}) = \frac{(2t - 1)R_i(x) + 2S'_i(x)}{2t -1} y^{-2(t-1)} \frac{\diff x}{2y} \in W_{-1,t-1}\text{.}
\end{equation*}
\end{lemma}
\begin{proof}\hypertarget{proof-308}{}
\hypertarget{p-2945}{}%
We have that%
\begin{gather}
x^{i} y^{-2t}\frac{\diff x}{2y} = \left(R_i(x) Q(x) + S_i(x)Q'(x)\right)y^{-2t}\frac{\diff x}{2y}= R_i(x) y^{-2t + 2}\frac{\diff x}{2y}+ S_i(x)y^{-2t}\diff y\text{,}\notag
\end{gather}
and also that \(\diff(S_i(x)y^{-2t + 1}) = S'_i(x)  y^{-2t + 1} \diff x + (-2t + 1 ) S_i(x) y^{-2t} \diff y\). Therefore by subtracting \(\frac{1}{-2t + 1} \diff(S_i(x)y^{-2t + 1})\) from \(x^{i} y^{-2t}\diff x/2y\), we see that%
\begin{align}
x^{i} y^{-2t}\frac{\diff x}{2y} \amp\sim R_i(x) y^{-2t + 2}\frac{\diff x}{2y} + \frac{1}{2t- 1} S'_i(x)  y^{-2t + 1} \diff x\notag\\
\amp = \frac{(2t - 1)R_i(x) + 2S'_i(x)}{2t -1} y^{-2(t-1)} \frac{\diff x}{2y}\text{.}\label{mrow-262}
\end{align}
\end{proof}
\typeout{************************************************}
\typeout{Subsection 47.4.4 Towards a faster algorithm}
\typeout{************************************************}
\subsection[{Towards a faster algorithm}]{Towards a faster algorithm}\label{subsec-towards-a-faster}
\hypertarget{p-2946}{}%
In order to make use of the same linear recurrence techniques as Harvey, we express the reduction process as we descend through the indices \(s,\,t\) as a linear recurrence with coefficients linear polynomials in \(s,\,t\). We describe such a recurrence that retains enough information to compute Coleman integrals. By working with a number of evaluations of the primitives on prescribed points on the curve, rather than the primitives themselves as power series, we only have to deal with a vector of fixed size at each step. This is preferable to maintaining a power series as we reduce, adding terms at each step.%
\par
\hypertarget{p-2947}{}%
We will now give an idea of the approach, giving the details in the next section. Let us first consider the end result of one row of the horizontal reduction process. Fixing a row \(t\), after the reduction we have an equality of the form%
\begin{equation}
\sum_{s\ge 0} a_s x^s y^{-2t}\frac{\diff x }{2y} -  \diff\left(\sum_{s \ge 0} c_sx^sy^{-2t +1} \right)  = \sum_{i = 0}^{2g-1} m_i x^i y^{-2t} \frac{\diff x}{2y}\in W_{-1,t}\label{men-16}
\end{equation}
in which the terms of the exact differential were found in decreasing order as the reductions are performed. Unfortunately, adding each new term as it is obtained is not a \emph{linear} recurrence in the index \(s\), as we have \(s\) appearing in the exponent of \(x\) in each term. Instead we observe that we can express the exact differential as%
\begin{equation}
\diff\left(( c_{0} + x(c_1 + x(\cdots +x(c_r))))y^{-2t + 1} \right)\text{.}\label{men-17}
\end{equation}
In essence, we are applying the \emph{Horner scheme} for polynomial evaluation.%
\par
\hypertarget{p-2948}{}%
Now we specialise to the case of computing the evaluation \(f_i(P)\) of the primitive for some point \(P = (x(P), y(P))\). We can, at each step, compute a further bracketed term starting from the innermost; using the given \(x, y\) values, we get a recurrence whose final term is the same as the original evaluation. So we can compute the terms of a recurrence of the form%
\begin{equation}
f_{i,0}= 0,\,f_{i, n} = x(P)f_{i, n-1}  - \frac{1}{(2t - 1)(2g + 1) -2s} d_{i,n}\label{men-18}
\end{equation}
where \(s = s_{\max} - n\) decreases from its maximum value, and \(d_{i,n}\) is the coefficient of the monomial removed in the \(n\)th step of the reduction process. Multiplying the result of this recurrence by the factor \(y^{-2t+1}\) (which is constant along the row) will result in the evaluation of the primitive for the row. At each step we will no longer have an evaluation of the primitive so far, it is only after completing all the reduction steps that each term will have the correct power of \(x\).%
\par
\hypertarget{p-2949}{}%
We may use the same technique for the vertical reductions; here we have%
\begin{equation*}
\sum_{t\ge 0} \sum_{i = 0}^{2g-1} m_i x^i y^{-2t} \frac{\diff x}{2y} -  \diff\left(\sum_{t \ge 1} \sum_{i = 0}^{2g}d_{ti} S_i(x)y^{-2t +1} \right)  = \sum_{i = 0}^{2g-1} M_i x^i \frac{\diff x}{2y}\in W_{-1,0}\text{,}
\end{equation*}
where now writing \(d_t = \sum_{i=0}^{2g-1} d_{t,i}S_i(x)\), the exact differential can be expressed as%
\begin{equation}
\diff\left(y^{-1}( d_{1} + y^{-2}(d_2 + y^{-2}(\cdots (d_{r-1} +y^{-2}(d_r))\cdots)))\right)\text{.}\label{men-19}
\end{equation}
\begin{remark}[]\label{remark-83}
\hypertarget{p-2950}{}%
The factor \(y^{-2t+1}\) appears in every term in the primitive in row \(t\). It is the same factor in the primitive for the vertical reduction from row \(t\) to row 0. So we can initialise the vertical recurrence from \(W_{-1,t}\) with both the differential and the evaluations obtained from horizontal reduction along row \(t\), and let the vertical reduction steps multiply the evaluation of the row primitives by this factor.%
\end{remark}
\hypertarget{p-2951}{}%
Now we write down the recurrences for both horizontal and vertical reductions precisely using matrices acting on appropriate vector spaces.%
\typeout{************************************************}
\typeout{Subsection 47.4.5 The recurrence}
\typeout{************************************************}
\subsection[{The recurrence}]{The recurrence}\label{subsection-106}
\hypertarget{p-2952}{}%
We will now switch to working with a \(\QQ_p\) vector \(h^t(s)\in W_{s,t}\times \QQ_p^L\) (resp. \(v(t) \in W_{-1,t} \times \QQ_p^L\)); these are of length \(2g+1 + L\) (resp. \(2g + L\)) in the horizontal case (resp.\ vertical case). The first entries represent the current differential we have reduced to, with respect to the basis given in \hyperref[eqn-basis-Wst]{(\ref{eqn-basis-Wst})}. The last \(L\) entries will contain the evaluations of the terms of the primitive picked up so far, one for each of the \(L\) points \(P_1,\ldots, P_L\in X(\QQ_p)\) we want evaluations at.%
\par
\hypertarget{p-2953}{}%
When we horizontally reduce, using the result of \hyperref[lem-harvey-horizontal-reduction]{Lemma~\ref{lem-harvey-horizontal-reduction}}, the two terms we are interested in, the exact differential and the reduction, have a common denominator of \(D_H^t(s) =(2t-1)(2g+1) -2s\). Similarly, in the vertical case, the two terms of interest in \hyperref[lem-harvey-vertical-reduction]{Lemma~\ref{lem-harvey-vertical-reduction}} have a common denominator of \(D_V(t) = 2t -1\).%
\par
\hypertarget{p-2954}{}%
Writing out the result of a single reduction step in terms of these vectors, we see that we need to compute the terms of the recurrence given by \(h^t(s) = R^t_H(s + 1)h^t(s+1)\) in the horizontal case, for \(R^t_H(s)\) defined by%
\begin{equation}
D_H^t(s)R_H^t(s) = M_H^t(s) =
\left(\begin{smallermatrix}{cccc|ccc}
0 \amp \cdots \amp 0\amp  p^t_0\amp       \amp      \amp               \\
D_H^t(s) \amp \cdots \amp 0\amp  p^t_1\amp       \amp      \amp               \\
\vdots\amp\ddots\amp\vdots\amp  \vdots\amp       \amp      \amp               \\
0 \amp \cdots \amp D_H^t(s)\amp  p^t_{2g}\amp       \amp      \amp               \\
\hline % -------------------------------------------------
0  \amp \cdots \amp 0\amp -1\amp x(P_1)D_H^t(s)\amp\cdots\amp 0             \\
\vdots\amp\ddots\amp\vdots\amp  \vdots      \amp \vdots\amp\ddots\amp \vdots        \\
0 \amp\cdots  \amp 0\amp -1\amp  0    \amp\cdots\amp x(P_L)D_H^t(s)        \\
\end{smallermatrix}\right)\text{,}\label{eqn-horiz-red-mat}
\end{equation}
where \(p^t_i\) is the linear function of \(s\) obtained as the coefficient of \(x^i\) in \(2sP(x) - (2t- 1)xP'(x)\). To divide through by \(D_H^t(s)\) we must multiply some terms by \(D_H^t(s)\).%
\par
\hypertarget{p-2955}{}%
For the vertical reductions we use \(R_V(t)\) defined by%
\begin{align}
\amp D_V(t)R_V(t) = M_V(t) =\notag\\
\amp \left(\begin{smallermatrix}{ccc|ccc}
(2t - 1)r_{0,0} + 2s'_{0,0}\amp \cdots \amp (2t - 1)r_{2g-1,0} + 2s'_{2g-1,0}\amp       \amp      \amp               \\
\vdots\amp\ddots\amp\vdots\amp        \amp       \amp                \\
(2t - 1)r_{0,2g-1} + 2s'_{0,2g-1} \amp \cdots \amp (2t - 1)r_{2g-1,2g-1} + 2s'_{2g-1,2g-1}\amp       \amp      \amp               \\
\hline % -------------------------------------------------
-S_0(x(P_1))  \amp \cdots \amp -S_{2g-1}(x(P_1))\amp y(P_1)^{-2}D_V(t)\amp\cdots\amp 0             \\
\vdots\amp\ddots\amp\vdots\amp  \vdots  \amp\ddots\amp \vdots        \\
-S_{0}(x(P_L)) \amp\cdots  \amp -S_{2g-1}(x(P_L))\amp 0    \amp\cdots\amp y(P_L)^{-2}D_V(t)
\end{smallermatrix}\right)\text{,}\label{mrow-264}
\end{align}
where \(r_{i,j}\) is the coefficient of \(x^j\) in \(R_i(x)\) and \(s'_{i,j}\) is the coefficient of \(x^j\) in \(S'_i(x)\). Once again we have multiplied the rightmost block by \(D_V(t)\) to extract the common denominator. We do this to express the reduction steps as linear recurrences with linear polynomial coefficients, rather than rational function coefficients.%
\par
\hypertarget{p-2956}{}%
Introducing the notation \(M^t_H(a,b) = M^t_H(a+1) \cdots M^t_H(b-1)M^t_H(b)\) (and the analogous \(M_V(a,b)\)), we can write the upshot of the above as%
\begin{theorem}[{}]\label{thm-coleman-harvey-reductions}
\hypertarget{p-2957}{}%
    Let \(h^t(s) = (\omega,0) \in W_{s,t}\times \QQ_p^L\), and \(f\colon X(\QQ_p)\to \QQ_p\), write \(c(f)\) for the \terminology{correction factor}, the linear endomorphism of \(W_{s,t}\times \QQ_p^L\) that is the identity on \(W_{s,t}\) and scales each component of \(\QQ_p^L\) by \(f(P_\ell)\) for the corresponding \(P_\ell\). Then the reduced vector \(c(y^{-1})R_V(0,t)c(y^{2})R_H^t(-1,s)h^t(s) \in W_{-1,0} \times \QQ_p^L\) is such that the projection onto \(W_{-1,0}\) is some \(\tilde\omega\) with \(\tilde\omega = \omega - \diff(g)\) for some \(g\in A^\dagger\), and the projection onto \(\QQ_p^L\) is \((g(P_1),\ldots ,g(P_L))\).%
\end{theorem}
\hypertarget{p-2958}{}%
As the approximation in \hyperref[eqn-approx-frobomegai]{(\ref{eqn-approx-frobomegai})} has summands that occur in several different \(W_{s,t}\)'s, we cannot simply find the product matrix and apply it to a single vector. Instead, we must work through the various subspaces doing as many reductions as possible before we reach a new monomial from the original approximation. As \(D_H\) and \(D_V\) are scalar matrices, we can commute them past the \(M_V\)'s and \(M_H\)'s. This  separates out the components so we can work just with products of matrices of linear polynomials. This reduces the problem to finding several of the products \(M_V(a,b)\) and \(M_H^t(a,b)\). In practice, to use as little \(p\)-adic precision as we can, we must not allow too many runs of multiplications by \(p\) and then divisions by \(p\), so that the relative precision stays as large as possible. This will be addressed in \hyperref[sec-coleman-harvey-precision-correctness]{\(\S\)\ref{sec-coleman-harvey-precision-correctness}}.%
\typeout{************************************************}
\typeout{Section 47.5 Linear recurrence algorithms}
\typeout{************************************************}
\section[{Linear recurrence algorithms}]{Linear recurrence algorithms}\label{sec-coleman-harvey-linear}
\hypertarget{p-2959}{}%
In this section, we recall and adapt some methods for finding subsequent terms of linear recurrences with linear polynomial coefficients. The set-up is that we are given an \(m\times m\) matrix \(M(x)\) with entries that are linear polynomials over a ring \(R\) and wish to obtain several products \(M(x,y) = M(y)M(y-1) \cdots M(x+1)\) for \(x \lt y\) integers. We let \(\mathrm{MM}(m, n)\) be the number of ring operations used when multiplying an \(n\times m\) matrix by an \(m\times m\) matrix, both with entries in \(R\). Then \(\mathrm{MM}(m) = \mathrm{MM}(m,m)\) is the cost of multiplying two \(m \times m\) matrices. We will not say much about these functions here, as modern theoretical bounds for these functions do not affect the point of our main result; however see \cite{LeGall2012} for some recent work on the topic. Using naive matrix multiplication, we have \(\mathrm{MM}(m,n) = O(m^2n)\), which if \(m^2 = o(n)\), cannot be improved upon asymptotically. Whenever \(n \ge m\) we can partition an \(n\times m\) matrix into roughly \(n/m\) blocks each of size \(m\times m\). These blocks can then be multiplied individually for a run time of \(\mathrm{MM}(m, n) = O(\mathrm{MM}(m)\frac{n}{m})\). We will also let \(\mathrm{M}(n)\) be the number of ring operations needed to multiply two polynomials of degree \(n\) with coefficients in \(R\).%
\par
\hypertarget{p-2960}{}%
The method of Bostan-Gaudry-Schost requires that certain elements of \(R\) be invertible. Moreover, they assume as input a product \(\mathrm D(\alpha,\beta,k)\) of several of these inverses. We will apply these methods in \(\ZZ/p^N\ZZ\) where the cost of computing inverses is negligible compared to the rest of the algorithm, so we will take this step for granted; see \cite{Bostan2007} for more details.%
\par
\hypertarget{p-2961}{}%
With the above set-up, Harvey, \cite{Harvey2007} Theorem 6.2, adjusts the algorithm of Bostan-Gaudry-Schost, \cite{Bostan2007} Theorem 15, to prove the following theorem:%
\begin{theorem}[{}]\label{thm-harvey-bgs}
\hypertarget{p-2962}{}%
Let \(M(x)\) be a \(m \times m\) matrix with entries that are linear polynomials in \(R\lb x \rb\), let \(0 \le K_1 \lt L_1 \le  K_2 \lt L_2 \le \cdots \le K_r \lt L_r \le K\) be integers, and let \(s = \lfloor \log_4 K \rfloor \). Suppose that \(2,3,\ldots, 2^s + 1\) are invertible in \(R\). Suppose also that \(r \lt K^{\frac 12 - \epsilon}\), with \(0 \lt \epsilon \lt 1/2\). Then \(M(K_1, L_1), \ldots, M(K_r, L_r)\) can be computed using \(O(\mathrm{MM}(m) \sqrt{K} + m^2 \mathrm{M} (\sqrt{K}))\) ring operations  in  \(R\).%
\end{theorem}
\hypertarget{p-2963}{}%
In order to apply this theorem to the above recurrences for computing the Coleman data, we introduce a variant better suited to the recurrences we obtained in \hyperref[subsec-towards-a-faster]{\(\S\)\ref{subsec-towards-a-faster}}. If we simply applied the same algorithm/result as Harvey naively, we would not get as good a run time in general.%
\begin{theorem}[{}]\label{thm-harvey-bgs-block}
\hypertarget{p-2964}{}%
With the same set-up as \hyperref[thm-harvey-bgs]{Theorem~\ref{thm-harvey-bgs}}, except that now let \(M(x)\) be instead an \((m+n) \times (m+n)\) block lower triangular matrix with 4 blocks, with top left block an \(m \times m\) matrix and bottom right block a diagonal matrix:%
\begin{equation}
\left(\begin{array}{c|c}
%\amp \amp \amp       \amp      \amp               \\
A \amp      0               \\
% \amp \amp \amp       \amp      \amp               \\
\hline % -------------------------------------------------
%\amp \amp \amp    d_1   \amp      \amp               \\
B\amp   \begin{smallermatrix}{ccc}d_1 \amp \amp \\  \amp \ddots \amp \\  \amp  \amp  d_n \end{smallermatrix}             \\
%\amp \amp \amp       \amp      \amp d_n              \\
\end{array}\right)\text{.}\label{men-21}
\end{equation}
 Then the interval products \(M(K_1, L_1), \ldots, M(K_r, L_r)\) can be computed using only \(O\left((\mathrm{MM}(m) + \mathrm{MM}(m,n)) \sqrt{K} + (m^2 + mn) \mathrm{M} (\sqrt{K})\right)\) ring operations  in  \(R\).%
\end{theorem}
\begin{proof}\hypertarget{proof-309}{}
\hypertarget{p-2965}{}%
The algorithm to do this is the same as the one given for \hyperref[thm-harvey-bgs]{Theorem~\ref{thm-harvey-bgs}} in \cite{Harvey2007} Theorem 6.2, only adjusted to take advantage of the fact that the matrices used are of a more restricted form as follows:%
\par
\hypertarget{p-2966}{}%
First, note that a product of matrices of the assumed form is again of the same shape, so one can work only with matrices of this form throughout. Such matrices should then be stored without keeping track of the entries that are always 0, as a pair of matrices \(A,B\) of size \(m\times m\) and \(n \times m\) respectively, and a list containing the \(n\) bottom right diagonal entries. Now the algorithm of Harvey and Bostan-Gaudry-Schost should be applied using this fixed representation.%
\par
\hypertarget{p-2967}{}%
The complexity of this algorithm is dominated by two main subtasks: shifting evaluations of the matrices and matrix multiplication. During the shifting step, we need only interpolate the non-zero entries; there are \((m+n)m + n\) of these. The number of ring operations required for this is then \(O((m^2 + mn)\mathrm{M}(\sqrt K))\).%
\par
\hypertarget{p-2968}{}%
For the matrix multiplication steps, the restricted form of the matrix once again allows us to use a specialised matrix multiplication routine. Here we can evaluate the block matrix product more efficiently, multiplying only the non-zero blocks, and using the fact that multiplying an \(n \times m\) matrix on the right by a square diagonal matrix stored as a list uses only \(O(nm)\) operations. Therefore the total complexity of multiplying two matrices of this form is \(O(\mathrm{MM}(m,n)+ \mathrm{MM}(m))\). As we do not modify the algorithm in any other way, the result follows.%
\end{proof}
\hypertarget{p-2969}{}%
The conditions on the matrix in \hyperref[thm-harvey-bgs-block]{Theorem~\ref{thm-harvey-bgs-block}} are precisely those satisfied by the matrices \(M_H^t(s)\) and \(M_V(t)\) from \hyperref[sec-coleman-harvey-reduction]{\(\S\)\ref{sec-coleman-harvey-reduction}}. So we may use this algorithm for computing block horizontal and vertical reductions for certain intervals.%
\begin{remark}[]\label{remark-84}
\hypertarget{p-2970}{}%
As well as utilising the polynomial structure of our matrices, for any row with sufficiently many terms compared to the desired precision, it is also possible to interpolate \(p\)-adically. This idea is due to Kedlaya and is explained in \cite{Harvey2007} Section 7.2.1. Using this allows us to compute fewer interval products using \hyperref[thm-harvey-bgs-block]{Theorem~\ref{thm-harvey-bgs-block}} by interpolating the remaining ones.%
\end{remark}
\hypertarget{p-2971}{}%
If we could compute to infinite precision, it would be optimal to reduce as far as possible at each reduction step, i.e.\@, until we get to index of a new term that needs adding. However, in practice, we should divide by \(p\) as soon as possible, in order to reduce the number of extra \(p\)-adic digits needed throughout. Therefore analysing when divisions by \(p\) occur informs which interval products are found.%
\typeout{************************************************}
\typeout{Section 47.6 The algorithm}
\typeout{************************************************}
\section[{The algorithm}]{The algorithm}\label{sec-coleman-harvey-algorithm}
\hypertarget{p-2972}{}%
In this section we describe the complete algorithm derived in the previous sections. The flow of the algorithm is the same as that of Harvey, only we use our larger matrices throughout and have to make some small adjustments to the evaluations. Care should be taken in all steps where division occurs, see \hyperref[sec-coleman-harvey-precision-correctness]{\(\S\)\ref{sec-coleman-harvey-precision-correctness}}.%
\begin{algorithm}[{Computation of Coleman data}]\label{algo-main}
\hypertarget{p-2973}{}%
Input: A list of points \(\{P_\ell\}_{1\le \ell \le L}\) in non-Weierstrass residue disks, precision \(N\).%
\par
\hypertarget{p-2974}{}%
Output: Matrix of Frobenius \(M\), modulo \(p^N\), such that \(\omega_i = \diff f_i +  \sum_j M_{ij} \omega_j\), evaluations \(f_i(P_\ell)\) modulo \(p^N\) for all \(i, \ell\) also.\leavevmode%
% TODO
\begin{enumerate}
\item\hypertarget{li-605}{}\hypertarget{p-2975}{}%
For each row index \(t = (p(2j + 1) - 1)/2\) for \(0 \le j \le N-1\) do:%
\begin{enumerate}
\item\hypertarget{li-606}{}\hypertarget{p-2976}{}%
Compute the horizontal reduction matrices \(M_H^t((k-1)p, kp-2g-2)\) and \(D^t_H((k-1)p, kp-2g-2)\) for \(0 \le k \le (2g+1)(j+1)-1\) using \hyperref[thm-harvey-bgs-block]{Theorem~\ref{thm-harvey-bgs-block}}, and the \(p\)-adic interpolation outlined in \cite{Harvey2007}  7.2.1, for \(k > N\).%
\item\hypertarget{li-607}{}\hypertarget{p-2977}{}%
For each basis differential \(\omega_i\), \(0 \le i \le 2g-1\) do:%
\begin{enumerate}
\item\hypertarget{li-608}{}\hypertarget{p-2978}{}%
Initialise a vector \(h_{ij} \in {(\ZZ/p^{N+1}\ZZ)}^{2g + 1 + L}\)%
\item\hypertarget{li-609}{}\hypertarget{p-2979}{}%
For each column index \(s = p(i+r+ 1) - 1\) for \(r = (2g+1)j\) down to \(0\) do:%
\begin{enumerate}
\item\hypertarget{li-610}{}\hypertarget{p-2980}{}%
Add the \(x^s{y}^{-2t}\) term of \hyperref[eqn-approx-frobomegai]{(\ref{eqn-approx-frobomegai})} to \(h_{ij}\).%
\item\hypertarget{li-611}{}\hypertarget{p-2981}{}%
Set \(h_{ij} = R_H^t(kp-2g -2, kp)h_{ij}\) by doing \(2g+2\) matrix-vector products.%
\item\hypertarget{li-612}{}\hypertarget{p-2982}{}%
Set \(h_{ij} = R_H^t((k-1)p,kp-2g-2)h_{ij}\).%
\item\hypertarget{li-613}{}\hypertarget{p-2983}{}%
Set \(h_{ij} = R_H^t((k-1)p)h_{ij}\).%
\end{enumerate}
\end{enumerate}
\end{enumerate}
\item\hypertarget{li-614}{}\hypertarget{p-2984}{}%
Initialise a \(2g \times L\) matrix for the evaluations \(E\) and a \(2g\times 2g\) matrix for the action of Frobenius \(M\).%
\item\hypertarget{li-615}{}\hypertarget{p-2985}{}%
    Compute the vertical reduction matrices \(M_V(0, (p-1)/2), M_V((p-1)/2 + jp, (p-1)/2 + (j+1)p)\) for \(1\le j \lt N\) and the corresponding \(D_V(t)\)'s to precision \(p^{N+1}\) using \hyperref[thm-harvey-bgs-block]{Theorem~\ref{thm-harvey-bgs-block}}, and divide through to obtain the corresponding \(R_V\)'s, label them \(R_j\).%
\item\hypertarget{li-616}{}\hypertarget{p-2986}{}%
For each basis differential \(\omega_i\), \(0\le i \le 2g-1\):%
\begin{enumerate}
\item\hypertarget{li-617}{}\hypertarget{p-2987}{}%
Initialise a zero vector \(v_{i} \in (\ZZ/p^N\ZZ)^{2g + L}\).%
\item\hypertarget{li-618}{}\hypertarget{p-2988}{}%
For each row index \(t = (p(2j + 1) - 1)/2\) for \(j = N- 1\) down to \(0\) do:%
\begin{enumerate}
\item\hypertarget{li-619}{}\hypertarget{p-2989}{}%
Add the last \(2g + L\) entries of \(h_{ij}\) to \(v_i\), correcting the last \(L\) entries as in \hyperref[thm-coleman-harvey-reductions]{Theorem~\ref{thm-coleman-harvey-reductions}}.%
\item\hypertarget{li-620}{}\hypertarget{p-2990}{}%
Set \(v_i = R_jv_i\).%
\end{enumerate}
\item\hypertarget{li-621}{}\hypertarget{p-2991}{}%
Set the \(i\)th column of \(M\) to be the first \(2g\) entries of \(v_i\).%
\item\hypertarget{li-622}{}\hypertarget{p-2992}{}%
Set the \(i\)th row of \(E\) to be the last \(L\) entries of \(v_{i}\), correcting them to be evaluations as in \hyperref[thm-coleman-harvey-reductions]{Theorem~\ref{thm-coleman-harvey-reductions}}.%
\end{enumerate}
\item\hypertarget{li-623}{}\hypertarget{p-2993}{}%
Output the matrix of Frobenius \(M\) and the matrix of evaluations \(E\).%
\end{enumerate}
\end{algorithm}
\begin{remark}[]\label{remark-85}
\hypertarget{p-2994}{}%
We have not  used the fact that in \hyperref[algo-coleman]{Algorithm~\ref{algo-coleman}} we only needed to evaluate at Teichmüller points. Using Teichmüller points only serves to make the description of Coleman integration a little simpler, and provides a convenient set of points corresponding to residue disks. This allows one to store the output of \hyperref[algo-main]{Algorithm~\ref{algo-main}} for further computations involving the same set of residue disks.%
\par
\hypertarget{p-2995}{}%
One simpler variant of this algorithm is to compute evaluations for one point at a time, re-running the whole procedure including finding the matrix of Frobenius once for each point. The advantage of this method is not needing a specialised version of the linear recurrence algorithms as in \hyperref[thm-harvey-bgs-block]{Theorem~\ref{thm-harvey-bgs-block}}. While this would result in the same theoretical run time if \(g^2 \in o(p)\), recomputing the matrix of Frobenius would be a duplication of work and inefficient in many parameter ranges.%
\end{remark}
\typeout{************************************************}
\typeout{Section 47.7 Precision}
\typeout{************************************************}
\section[{Precision}]{Precision}\label{sec-coleman-harvey-precision-correctness}
\hypertarget{p-2996}{}%
In this section we examine the level of \(p\)-adic precision that needs to be maintained throughout, in order to compute the matrix of Frobenius and evaluations of primitives to precision \(O(p^N)\). We follow Harvey's approach in \cite{Harvey2007} Section 7 and prove that analogous results hold for our recurrence.%
\begin{lemma}[{}]\label{lemma-66}
\hypertarget{p-2997}{}%
During horizontal reduction, the evaluations of the primitives remain integral. Moreover, if the calculations are performed with initial data known to absolute precision \(p^N\) and intermediate computations are performed with an absolute precision cap of \(p^{N+1}\), then whenever division by \(p\) occurs, the dividend is known to absolute precision \(p^{N+1}\), so that the quotient is known to absolute precision \(O(p^N)\).%
\end{lemma}
\begin{proof}\hypertarget{proof-310}{}
\hypertarget{p-2998}{}%
As we begin with evaluation 0, we must show that if the evaluations are integral, they remain so after several reduction steps. Any point \(P = (x,y)\) that we are evaluating at is assumed not be in a  Weierstrass residue disk and in particular not in the residue disk at infinity. Hence \(x\) is integral and multiplication by it will never reduce \(p\)-adic valuation.%
\par
\hypertarget{p-2999}{}%
In the horizontal reduction matrix \hyperref[eqn-horiz-red-mat]{(\ref{eqn-horiz-red-mat})}, the only nonzero terms in the bottom left block are the \(-1\)s in the rightmost column which will not disturb integrality.%
\par
\hypertarget{p-3000}{}%
    When \(D_H^t(s)\equiv 0 \pmod p\), it is shown in \cite{Harvey2007} Claim 7.3,  using the assumptions on \(p\) in \hyperref[eqn-assum-p]{(\ref{eqn-assum-p})}, that the vector currently being reduced has its \((2g+1)\)-component divisible by \(p\) and is correct to absolute precision \(p^{N+1}\). Thus this can be divided by \(D_H^t(s)\) while keeping absolute precision \(p^N\). Every column of \(M_H^t(s)\) other than the \((2g+1)\)st has \(D_H^t(s)\) as a factor, so the division can be performed.%
\par
\hypertarget{p-3001}{}%
All other steps follow directly from the work of Harvey.%
\end{proof}
\begin{lemma}[{}]\label{lemma-67}
\hypertarget{p-3002}{}%
During vertical reduction, the evaluations of the primitives remain integral. Moreover, if the calculations are performed with initial data known to absolute precision \(p^N\) and intermediate computations are performed with an absolute precision cap of \(p^{N+1}\), then whenever division by \(p\) occurs, the dividend is known to absolute precision \(p^{N+1}\), so that the quotient is known to absolute precision \(O(p^N)\).%
\end{lemma}
\begin{proof}\hypertarget{proof-311}{}
\hypertarget{p-3003}{}%
Any point \(P = (x,y)\) that we are evaluating at is assumed not to be in a  Weierstrass residue disk and in particular not in the residue disk at infinity. Hence \(y\) is a unit and multiplying or dividing by it will not change \(p\)-adic valuation.%
\par
\hypertarget{p-3004}{}%
We check that the analysis in \cite{Harvey2007} Lemmas 7.7 and 7.9 may be adjusted to apply with our extended \(M_V(t)\). Assume that \(t \equiv 1/2 \pmod p\) so that \(D_V(t) \equiv 0 \pmod{p}\), in this case \(v_p(D_V(t))= 1\) as \hyperref[eqn-assum-p]{(\ref{eqn-assum-p})} implies \(D_V(t) \lt p^2\). Unlike in \cite{Harvey2007} Lemma 7.7, our matrix \(M_V(t)\) will not have integral inverse as \(D_V(t)\) appears in the bottom right block, so \(M_V(t)\) is singular mod \(p\). Instead, the inverse of the block lower triangular \(M_V(t)\) has integral top left block, and the bottom two blocks have valuation at least \(-1\). Now letting \(t_0 = (p-1)/2\) and  \(X = D_V(t_0, t_0 + p + 1)^{-1}M_V(t_0, t_0 + p + 1)\), the argument in \cite{Kedlaya2001} Lemma 2 implies that \(pX\) is integral. The argument says that taking \(\omega\in W_{-1,t_0+p+1}\) with integral coefficients, the primitive \(g\) of \(X\omega - \omega\) becomes integral after multiplication by \(p\), and hence the evaluation of \(pg\) at a point in a non-Weierstrass residue disk is integral. The entries in the bottom left block of \(X\) are evaluations of this form up to a power of \(y(P)\), which will not affect integrality. The bottom right block of \(X\) is integral already as it is simply a power of the diagonal matrix \(\operatorname{diag}((y(P_\ell)^{-2})_\ell)\). So each term of the block matrix product \((pX)M_V(t_0 + p + 1)\) is integral, and \(M_V(t_0, t_0 + p) = D_V(t_0, t_0 + p + 1) X M_V(t_0 + p + 1)^{-1}\) is divisible by \(p\).%
\end{proof}
\begin{remark}[]\label{remark-86}
\hypertarget{p-3005}{}%
Multiplying by \((M-I)^{-1}\), as in \hyperref[eqn-coleman-teich-infty-system]{(\ref{eqn-coleman-teich-infty-system})}, will lose \(v_p(\det(M-I))\) digits of absolute \(p\)-adic precision. As \(v_p(\det(M-I)) = v_p(\Jac(X)(\FF_p)[p])\), this is at most \(g\) in the \terminology{anomolous case}, and in general we expect that it is 0, so if \( g = O(N)\) the whole computation can be repeated with the extra precision required  at no extra asymptotic cost.%
\end{remark}
\typeout{************************************************}
\typeout{Section 47.8 Run time analysis}
\typeout{************************************************}
\section[{Run time analysis}]{Run time analysis}\label{sec-coleman-harvey-analysis}
\hypertarget{p-3006}{}%
Having described the algorithm in detail, we now analyse its run time, in order to prove \hyperref[thm-coleman-harvey-main]{Theorem~\ref{thm-coleman-harvey-main}}. First of all we analyse each step of \hyperref[algo-main]{Algorithm~\ref{algo-main}}.%
\par
\hypertarget{p-3007}{}%
The main step is the computation of the reduction matrices via \hyperref[thm-harvey-bgs-block]{Theorem~\ref{thm-harvey-bgs-block}}. In this case, we have \(m = 2g\) (\(+1\) in the horizontal case) and \(n = L\). When reducing horizontally, for each row the largest index is bounded by \(K=O(Np)\). When reducing vertically our index is also at most \(O(Np)\). As there are \(N\) rows in total, we obtain a total of%
\begin{equation}
O\left(N((\mathrm{MM}(g) + \mathrm{MM}(g,L)) \sqrt{Np} + (g^2 + gL) \mathrm{M} (\sqrt{Np}))\right)\label{men-22}
\end{equation}
ring operations to compute the matrices. Using that \(\mathrm M (d) \in \widetilde O (d)\), that \(\mathrm{MM}(m) = m^\omega\) for some \(2\le \omega \le 3\), and the above discussion of \(\mathrm{MM}(m,n)\), we simplify to \(O\left((g^\omega + Lg^{\omega - 1}) \sqrt{p}N^{3/2}\right)\) ring operations, bit complexity \(\widetilde O\left((g^\omega + Lg^{\omega - 1}) \sqrt{p} N^{5/2}\right)\).%
\par
\hypertarget{p-3008}{}%
The remaining operations are exactly as analysed by Harvey in \cite{Harvey2007} Section 7.4. With our larger, but still sparse, horizontal reduction matrices, each reduction step without \hyperref[thm-harvey-bgs-block]{Theorem~\ref{thm-harvey-bgs-block}} uses \(O(g+L)\) rather than \(O(g)\) ring operations, for a total of \(O(N^3g^3(g+L))\) ring operations, or \(\widetilde O(N^4 g^3(g+L) \log p)\) bit operations. We then have a total time complexity of%
\begin{equation}
\widetilde O\left((g^\omega + Lg^{\omega - 1}) \sqrt{Np} N^2 + N^4 g^3(g+L) \log p \right)\text{.}\label{eqn-final-complexity}
\end{equation}
\par
\hypertarget{p-3009}{}%
Now we turn to the algorithm for computing Coleman integrals, obtained by running \hyperref[algo-main]{Algorithm~\ref{algo-main}} once and then \hyperref[algo-coleman]{Algorithm~\ref{algo-coleman}} once for each point. The analysis here is the same as that in \cite{Balakrishnan2010} Section 4.2, where, by using \hyperref[algo-main]{Algorithm~\ref{algo-main}} instead of Kedlaya's algorithm, we may replace the \(\widetilde O(pN^2g^2)\) in their complexity analysis with \hyperref[eqn-final-complexity]{(\ref{eqn-final-complexity})}. The remaining steps to complete the Coleman integration are logarithmic in \(p\) and are dominated by the logarithmic in \(p\) term of \hyperref[eqn-final-complexity]{(\ref{eqn-final-complexity})}.%
\par
\hypertarget{p-3010}{}%
If \(L\) is fixed (for example \(L=2\) when computing integrals between two points) the complexity is as in \cite{Harvey2007} Theorem 1.1. This finishes the proof of \hyperref[thm-coleman-harvey-main]{Theorem~\ref{thm-coleman-harvey-main}}.%
\begin{remark}[]\label{remark-87}
\hypertarget{p-3011}{}%
The version of Kedlaya's algorithm used in \cite{Balakrishnan2010} Algorithm 10, seems to have an advantage in that it outputs the power series of the \(f_i\)'s. This could of course be re-used later to evaluate at further points without re-running Kedlaya's algorithm. However, for \(p\) large enough, this series has so many terms that it is faster asymptotically to recompute everything with the algorithm given here, than it is to evaluate the power series at one point.%
\end{remark}
\typeout{************************************************}
\typeout{Section 47.9 Implementation}
\typeout{************************************************}
\section[{Implementation}]{Implementation}\label{sec-coleman-harvey-implementation}
\hypertarget{p-3012}{}%
We have implemented this algorithm in C\mono{++} as an extension of David Harvey's \mono{hypellfrob} package. This extension has been wrapped and can be easily used from within Sage \cite{Sagemath}. The implementation is included as part of the supplementary materials to this paper. This implementation uses naive matrix multiplication (for which \(\omega= 3\)) and does not take into account the special form of the matrices, as in \hyperref[thm-harvey-bgs-block]{Theorem~\ref{thm-harvey-bgs-block}}; so the run time of this implementation will not have the asymptotic behaviour stated in \hyperref[eqn-final-complexity]{(\ref{eqn-final-complexity})} for the parameter \(L\).%
\par
\hypertarget{p-3013}{}%
In \hyperref[tab-coleman-harvey-gen3-times]{Table~\ref{tab-coleman-harvey-gen3-times}}, we list some timings obtained using this implementation in genus 3, for various primes \(p\) and \(p\)-adic precision bounds \(N\). For comparison, we also list timings for the functionality for computing Coleman integrals in Sage 8.0. The implementation in Sage is written in Python, rather than C\mono{++}, so we would expect some speed-up even if a superior algorithm was not used. Specifically we have compared the time to compute the Coleman data only, and do not include any of the time spent doing the linear algebra and tiny integral steps of Coleman integration, which should be comparatively fast. As such, we only time the components that will differ between the old and new approaches. For the existing Sage code we have timed both finding the matrix of Frobenius and the primitives (by calling \mono{monsky\_washnitzer.matrix\_of\_frobenius\_hyperelliptic}), and the time to evaluate the resulting primitive at one point. This is compared with the time taken by the new implementation, called from its Sage wrapper with one point specified, this outputs the matrix of Frobenius and the evaluations at that point. All timings and examples are on a single 16 AMD Opteron 8384 2.7GHz processor on a  machine with 16 cores and 82 GB RAM\@. While this table is mostly intended to show practicality, in the \(N = 9\) column the square root dependence on \(p\) can be seen. The large jump in the timings between \(p \approx 256\) and \(p \approx 512\) for \(N = 7\) could be explained by the fact that this is the cut off between when an element of \(\ZZ/p^N\ZZ\) is representable in one machine word.%
\begin{table}
\centering
\begin{tabular}{lCllllll}
\(p\backslash N\)&1&3&5&7&9\tabularnewline\hrulethick
131&1.14/0.01&\phantom{0}3.67/0.02&\phantom{0}9.36/0.07&16.90/0.12&\phantom{0}20.06/0.49\tabularnewline[0pt]
257&1.96/0.01&\phantom{0}8.90/0.03&20.83/0.07&30.91/0.18&\phantom{0}63.14/0.68\tabularnewline[0pt]
521&4.73/0.01&19.23/0.03&39.18/0.08&86.49/0.62&162.81/0.91
\end{tabular}
\caption{Timings for genus 3: Sage 8.0 time/New time (sec)\label{tab-coleman-harvey-gen3-times}}
\end{table}
\typeout{************************************************}
\typeout{Section 47.10 Examples}
\typeout{************************************************}
\section[{Examples}]{Examples}\label{sec-coleman-harvey-examples}
\hypertarget{p-3014}{}%
In this section we give an explicit example of a computation we can perform with this technique, demonstrating how large we can feasibly take the parameters. We compare our implementation to the existing functionality for Coleman integration in Sage 8.0 for this example.%
\par
\hypertarget{p-3015}{}%
The current implementation uses the basis \(x^i \diff x/y\), to remain consistent with Harvey's notation. As the existing functionality for Coleman integration in Sage 8.0 uses the basis \(x^i\diff x / 2y\) for cohomology, we must divide the obtained evaluations by 2 to compare them to those returned by Sage or \hyperref[algo-main]{Algorithm~\ref{algo-main}}.%
\begin{example}[]\label{ex-leprevost-large-coleman}
\hypertarget{p-3016}{}%
Let \(C\colon y^2 = x^5 + \frac{33}{16} x^4 +\frac 34 x^3 + \frac 38 x^2 - \frac 14 x + \frac{1}{16}\) be Leprévost's curve, as in \cite{Balakrishnan2010} Example 21. Then letting \(P = (-1,1)\), \(Q = (0,\frac 14)\) and \(p = 2^{45} + 59 = 35184372088891\), using our implementation we can compute the matrix of Frobenius \(M\) to 1 \(p\)-adic digit of precision, and also that%
\begin{align*}
f_0(P) - f_0(Q) \amp= O(p), \amp\amp f_1(P) - f_1(Q) = O(p),\\
f_2(P) - f_2(Q) \amp=  7147166195043 + O(p), \amp\amp f_3(P) - f_3(Q) =  9172338112529 + O(p)\text{.}
\end{align*}
Computing this (and finding \((M-1)^{-1}\)) takes a total of 27.8 minutes (with a peak memory usage of 2.9GB). Evaluating Coleman integrals for such a large prime is far out of the range of what was possible to compute in reasonable amount of time using the previous implementation. In fact, even when \(p = 2^{14} + 27\) the existing Sage functionality takes 53.2 minutes, and uses a larger volume of memory (12GB).%
\par
\hypertarget{p-3017}{}%
As we have used only 1 digit of \(p\)-adic precision, the points \(P\) and \(Q\) are congruent up to this precision to the corresponding Teichmüller point in their residue disk. So, for this example, we do not need to worry about computing tiny integrals; the vector of Coleman integrals \(\int^P_Q \omega_i\) can be obtained from the above vector of evaluations by multiplying by \((M-1)^{-1}\). Doing this gives us the vector \((O(p), O(p), 9099406574713 + O(p), 7153144612900 + O(p))\) reflecting the holomorphicity of the first two basis differentials only. We have also run the same example with precision \(N = 3\); this took 22.5 hours and used a peak of 50GB of memory.%
\end{example}
\typeout{************************************************}
\typeout{Section 47.11 Future directions}
\typeout{************************************************}
\section[{Future directions}]{Future directions}\label{sec-coleman-harvey-future}
\hypertarget{p-3018}{}%
The assumptions on the size of \(p\) allow us to use at most one extra digit of \(p\)-adic precision; it should be possible to relax this assumption somewhat, using a more complicated algorithm instead. Similarly it should be possible to work over extensions of \(\QQ_p\), or remove the assumption that \(Q(x)\) is monic.%
\par
\hypertarget{p-3019}{}%
Kedlaya's algorithm has been generalised to other curves and varieties, e.g.\@ \cite{Harrison2012,Gaudry2001,Goncalves2015,Tuitman2017} and Harvey's techniques have also been generalised to some of these cases \cite{Minzlaff2010,ABCMT2018}. Moreover, explicit Coleman integration has also been carried out in some of these settings, for even degree hyperelliptic curves \cite{Balakrishnan2015}, and for general curves \cite{Balakrishnan2017a}. It would be interesting to adapt our techniques to those contexts. Iterated Coleman integrals are also of interest and have been made computationally effective \cite{Balakrishnan2013}. Extending the algorithm presented here to compute iterated integrals is another natural next step. Harvey has also  described an \emph{average polynomial time} algorithm for dealing with for many primes at once \cite{Harvey2014a}. The author plans to explore the feasibility of analogous techniques when computing Coleman integrals.%

%\nocite{*}

\bibliographystyle{ants}
\bibliography{colemanharvey}

\newcommand{\etalchar}[1]{$^{#1}$}
\begin{thebibliography}{ABC{\etalchar{+}}18}
\newcommand{\enquote}[1]{``#1''}
\providecommand{\url}[1]{\texttt{#1}}
\providecommand{\urlprefix}{URL }
\expandafter\ifx\csname urlstyle\endcsname\relax
  \providecommand{\doi}[1]{doi:\discretionary{}{}{}#1}\else
  \providecommand{\doi}{doi:\discretionary{}{}{}\begingroup
  \urlstyle{rm}\Url}\fi
\providecommand{\eprint}[2][]{\url{#2}}

\bibitem[ABC{\etalchar{+}}18]{ABCMT2018}
Arul V., Best A.J., Costa E., Magner R., Triantafillou N.
\newblock \enquote{Computing {Z}eta {F}unctions of {C}yclic {C}overs in {L}arge
  {C}haracteristic}.
\newblock In \enquote{A{NTS} {XIII}---{P}roceedings of the {T}hirteenth
  {A}lgorithmic {N}umber {T}heory {S}ymposium}, This volume. Math. Sci. Publ.,
  Berkeley, CA, 2018.

\bibitem[Bal13]{Balakrishnan2013}
Balakrishnan J.S.
\newblock \enquote{Iterated {C}oleman integration for hyperelliptic curves}.
\newblock In \enquote{A{NTS} {X}---{P}roceedings of the {T}enth {A}lgorithmic
  {N}umber {T}heory {S}ymposium}, volume~1 of \emph{Open Book Ser.}, pages
  41--61. Math. Sci. Publ., Berkeley, CA, 2013.
\newblock \doi{10.2140/obs.2013.1.41}.

\bibitem[Bal15]{Balakrishnan2015}
---{}---{}---.
\newblock \enquote{Coleman integration for even-degree models of hyperelliptic
  curves}.
\newblock \emph{LMS J. Comput. Math.}, 18(1):258--265, 2015.

\bibitem[BBK10]{Balakrishnan2010}
Balakrishnan J.S., Bradshaw R.W., Kedlaya K.S.
\newblock \enquote{Explicit {C}oleman integration for hyperelliptic curves}.
\newblock In \enquote{Algorithmic number theory}, volume 6197 of \emph{Lecture
  Notes in Comput. Sci.}, pages 16--31. Springer, Berlin, 2010.

\bibitem[BD18]{Balakrishnan2016}
Balakrishnan J.S., Dogra N.
\newblock \enquote{Quadratic Chabauty and rational points I: p-adic heights}.
\newblock \emph{Duke Mathematical Journal}, 2018.
\newblock \eprint{http://arxiv.org/abs/1601.00388v2}.

\bibitem[Bes12]{Besser2012}
Besser A.
\newblock \enquote{Heidelberg lectures on {C}oleman integration}.
\newblock In \enquote{The arithmetic of fundamental groups---{PIA} 2010},
  volume~2 of \emph{Contrib. Math. Comput. Sci.}, pages 3--52. Springer,
  Heidelberg, 2012.

\bibitem[BGS07]{Bostan2007}
Bostan A., Gaudry P., Schost E.
\newblock \enquote{Linear recurrences with polynomial coefficients and
  application to integer factorization and {C}artier-{M}anin operator}.
\newblock \emph{SIAM J. Comput.}, 36(6):1777--1806, 2007.

\bibitem[BT17]{Balakrishnan2017a}
Balakrishnan J.S., Tuitman J.
\newblock \enquote{Explicit Coleman integration for curves}.
\newblock 2017.
\newblock \eprint{1710.01673v2}.

\bibitem[CdS88]{Coleman1988}
Coleman R., de~Shalit E.
\newblock \enquote{{$p$}-adic regulators on curves and special values of
  {$p$}-adic {$L$}-functions}.
\newblock \emph{Invent. Math.}, 93(2):239--266, 1988.

\bibitem[Col82]{Coleman1982}
Coleman R.F.
\newblock \enquote{Dilogarithms, regulators and {$p$}-adic {$L$}-functions}.
\newblock \emph{Invent. Math.}, 69(2):171--208, 1982.

\bibitem[Col85]{Coleman1985}
---{}---{}---.
\newblock \enquote{Torsion points on curves and {$p$}-adic abelian integrals}.
\newblock \emph{Ann. of Math. (2)}, 121(1):111--168, 1985.

\bibitem[GG01]{Gaudry2001}
Gaudry P., G\"urel N.
\newblock \enquote{An extension of {K}edlaya's point-counting algorithm to
  superelliptic curves}.
\newblock In \enquote{Advances in cryptology---{ASIACRYPT} 2001 ({G}old
  {C}oast)}, volume 2248 of \emph{Lecture Notes in Comput. Sci.}, pages
  480--494. Springer, Berlin, 2001.

\bibitem[Gon15]{Goncalves2015}
Gon\c{c}alves C.
\newblock \enquote{A point counting algorithm for cyclic covers of the
  projective line}.
\newblock In \enquote{Algorithmic arithmetic, geometry, and coding theory},
  volume 637 of \emph{Contemp. Math.}, pages 145--172. Amer. Math. Soc.,
  Providence, RI, 2015.

\bibitem[Har07]{Harvey2007}
Harvey D.
\newblock \enquote{Kedlaya's algorithm in larger characteristic}.
\newblock \emph{Int. Math. Res. Not. IMRN}, (22):Art. ID rnm095, 29, 2007.

\bibitem[Har12]{Harrison2012}
Harrison M.C.
\newblock \enquote{An extension of {K}edlaya's algorithm for hyperelliptic
  curves}.
\newblock \emph{J. Symbolic Comput.}, 47(1):89--101, 2012.

\bibitem[Har14]{Harvey2014a}
Harvey D.
\newblock \enquote{Counting points on hyperelliptic curves in average
  polynomial time}.
\newblock \emph{Ann. of Math. (2)}, 179(2):783--803, 2014.

\bibitem[Ked01]{Kedlaya2001}
Kedlaya K.S.
\newblock \enquote{Counting points on hyperelliptic curves using
  {M}onsky-{W}ashnitzer cohomology}.
\newblock \emph{J. Ramanujan Math. Soc.}, 16(4):323--338, 2001.

\bibitem[LG12]{LeGall2012}
Le~Gall F.
\newblock \enquote{Faster algorithms for rectangular matrix multiplication}.
\newblock In \enquote{2012 {IEEE} 53rd {A}nnual {S}ymposium on {F}oundations of
  {C}omputer {S}cience---{FOCS} 2012}, pages 514--523. IEEE Computer Soc., Los
  Alamitos, CA, 2012.

\bibitem[Min10]{Minzlaff2010}
Minzlaff M.
\newblock \enquote{Computing zeta functions of superelliptic curves in larger
  characteristic}.
\newblock \emph{Math. Comput. Sci.}, 3(2):209--224, 2010.

\bibitem[{Sag}18]{Sagemath}
{Sage Developers, The}.
\newblock \emph{{S}ageMath, the {S}age {M}athematics {S}oftware {S}ystem
  ({V}ersion 8.0.0)}, 2018.
\newblock {\tt http://www.sagemath.org}.

\bibitem[Tui17]{Tuitman2017}
Tuitman J.
\newblock \enquote{Counting points on curves using a map to {$\bold{P}^1$},
  {II}}.
\newblock \emph{Finite Fields Appl.}, 45:301--322, 2017.

\end{thebibliography}

\end{document}